\documentclass[letterpaper,11pt,reqno]{amsart} 
\usepackage[portrait,margin=3cm]{geometry}
\usepackage{systeme}
\usepackage{mathrsfs,xfrac} 
\usepackage[colorlinks=true,linkcolor=blue,citecolor=blue,urlcolor=blue]{hyperref} 
\usepackage{amsmath,amssymb,amsthm,amsfonts,amsbsy,latexsym,dsfont,color} 
\usepackage[numeric,initials,nobysame]{amsrefs} 
\usepackage[foot]{amsaddr}

\usepackage[utf8]{inputenc}
\usepackage[english]{babel}
\usepackage{comment}

\usepackage[textsize=tiny]{todonotes}
\usepackage{regexpatch}
\makeatletter
\xpatchcmd{\@todo}{\setkeys{todonotes}{#1}}{\setkeys{todonotes}{inline,#1}}{}{}
\makeatother

\usepackage{enumerate}

\renewcommand{\le}{\leqslant} 
\renewcommand{\ge}{\geqslant}

         \let\gS=\Sigma  
                                
\newcommand{\cN}{\mathcal{N}}

\DeclareMathOperator{\E}{\mathds{E}}

\DeclareMathOperator{\N}{N}

\usepackage{enumitem}
\numberwithin{equation}{section}
\newtheorem{theorem}{Theorem}[section]
\newtheorem{proposition}[theorem]{Proposition}
\newtheorem{lemma}[theorem]{Lemma}
\newtheorem{remark}[theorem]{Remark}

\providecommand{\E}{\mathbb E}
\providecommand{\R}{\mathbb R}
\providecommand{\N}{\mathbb N}
\providecommand{d}{\,\mathrm d}
\providecommand{\SigmaNS}{\Sigma_{N,S}}
\providecommand{\Wtil}{\widetilde W}
\providecommand{\Dtil}{\widetilde D}

\makeatletter
\@namedef{subjclassname@2020}{\textup{2020} Mathematics Subject Classification}
\makeatother

\title[Local fields for the GS model]{Limiting Law of Local Fields for the Mean Field Ghatak-Sherrington Model}
\author[Chen]{Yunhui Chen$^\star$}
\author[Fierro]{Keaton Fierro$^\dagger$}

\address{School of Mathematics, University of Minnesota, 127 Vincent Hall 206 Church St. SE Minneapolis, MN 55455}
\email{$^\star$chen8668@umn.edu, $^\dagger$fierr036@umn.edu}
\date{\today}

\begin{document}

\begin{abstract}
We study the limiting law of local field in the mean field Ghatak-Sherrington model with spin values $\{0,\pm1,\ldots,\pm S\}$ in the high-temperature regime. Using the cavity method and the quantitative overlap and self-overlap estimates of Sheng-Wu~\cite{SW24}, we prove a central limit theorem for the non-centered cavity field. As a result, we then identify the local field limiting law as a random finite mixture of Gaussian distributions. The proof is based on moment method and gives a quantitative error bound.
\end{abstract}

\maketitle

\section{Introduction and Main Results}

Spin glass models are parameterized families of disordered Gibbs measures on high dimensional spaces, originally introduced to model magnetic alloys with competing random interactions. The classical mean field example is the Sherrington-Kirkpatrick (SK) model \cite{SherringtonKirkpatrick1975} on binary cube. Rigorous understanding of mean field spin glass theory has seen notable progress in the last several decades, particularly after Guerra's interpolation bound \cite{Guerra2003}, Talagrand's proof of the Parisi variational formula for the limiting free energy \cite{Talagrand2006}, and Panchenko's development of the general mean-field theory, including the Ghirlanda-Guerra identities and ultrametricity \cite{Panchenko2013,Panchenko2014}. These works give the modern rigorous form of the Parisi picture, which originated in the physics literature \cite{Parisi1979,MezardParisiVirasoro1987}.

Besides the Parisi variational approach, the Thouless-Anderson-Palmer (TAP) approach \cite{ThoulessAndersonPalmer1977,Plefka1982} offers a different perspective on spin glass theory. The basic idea is to introduce a high dimensional random function, known as the TAP free energy, whose critical points satisfy the celebrated TAP equations. These equations form a system of self-consistency relations for the local magnetizations and include the Onsager reaction term. For the SK model, rigorous TAP related results were treated in \cite{Talagrand2011,Chatterjee2010, Chen2013}, Bolthausen's iterative construction of TAP solutions \cite{Bolthausen2014} suggested an alternative algorithmic perspective instead. For more general mixed $p$-spin models and results beyond high temperature regime, see~\cite{AuffingerJagannath2019, ChenPanchenko2018,ChenTang2024}. It turns out that the TAP equations are closely related to the limiting behavior of the local field. The local field essentially describes the effective interaction experienced by a single spin from its neighbors. Understanding its distribution gives direct information about the structure of the Gibbs measure and is one of the standard routes to deriving TAP equations.

For the classical SK model, the limiting behavior of local fields has been studied extensively. Chatterjee \cite{Chatterjee2010} obtained quantitative results for cavity and local fields through Stein's method. Chen \cite{Chen2013} later established central limit theorems for both cavity and local fields in the high temperature regime by Gaussian interpolation. These results provide a rigorous description of the local field distribution in the SK model. There is also a useful multi-species version of this picture. In the multi-species Sherrington-Kirkpatrick model, the sites are partitioned into species $I_s$, $s\in\mathscr{S}$, with asymptotic proportions $\lambda_s$, and the variance of the coupling between a spin in species $s$ and a spin in species $t$ is encoded by a matrix $\Delta^2=(\Delta_{s,t}^2)_{s,t\in\mathscr{S}}$. The relevant overlap is then a vector $\boldsymbol{q}=(q_s)_{s\in\mathscr{S}}$, and high temperature overlap vector concentration leads to species-wise Onsager corrections \cite{DeyWu2020MSK,Wu2023MSKTAP}.

The Ghatak-Sherrington (GS) model \cite{GhatakSherrington1977} is a multi-spin extension of the SK model in which the spin variables may take multiple values and a crystal field parameter is naturally included. The model has also attracted attention in the physics literature because of its richer phase diagram, including inverse freezing phenomena, see \cite{CrisantiLeuzzi2005} and \cite{SchupperShnerb2005}. Yokota \cite{Yokota1992} proposed TAP equations for the GS model, and Auffinger-Chen \cite{AC21} later gave a rigorous proof in the high temperature regime at arbitrary crystal field. Sheng-Wu \cite{SW24} proved a central limit theorem for the overlap and self-overlap quantities in the same high temperature regime. In a related direction, Sheng \cite{Sheng2024} proved a central limit theorem for the free energy in the mean field GS model at high temperature. 

In this paper, we study the limiting law of local fields for the mean field GS model. Using the cavity method together with the overlap and self-overlap estimates established in Sheng-Wu \cite{SW24}, we obtain a quantitative central limit theorem for the cavity field and further derive the limiting distribution of the local field. Specifically, we prove that the limiting law of the local field in the GS model is a finite mixture of Gaussian distributions which admits an explicit form.

\subsection{Definition and Main Results}

Fix $S\ge 1$, and let $\sigma=(\sigma_1,\ldots,\sigma_N)\in \SigmaNS:=\{0,\pm1,\pm2,\ldots,\pm S\}^N$. The disorder $(g_{ij})_{1\le i<j\le N}$ are independent standard Gaussian random variables. The associated Hamiltonian for the Ghatak-Sherrington model is given as
\begin{equation}\label{eq:Hamiltonian}
  H_N(\sigma)=\frac{\beta}{\sqrt N}\sum_{i<j}g_{ij}\sigma_i\sigma_j
  +D\sum_{i=1}^N\sigma_i^2+h\sum_{i=1}^N\sigma_i,
\end{equation}
where $\beta>0$, $D\in\R$, and $h\in\R$ are respectively the inverse temperature, crystal field, and external field parameters. The free energy is
\begin{equation}\label{eq:free-energy}
  F_{N,S}(\beta,D,h):=\frac1N\log Z_{N,S}(\beta,D,h),
\end{equation}
where the partition function is
\[
  Z_{N,S}(\beta,D,h):=\sum_{\sigma\in\SigmaNS}\exp(H_N(\sigma)).
\]
The Gibbs measure of the GS model is
\begin{equation}\label{eq:gibbs-measure}
  d G(\sigma)=Z_{N,S}(\beta,D,h)^{-1}\exp(H_N(\sigma))\,d\sigma.
\end{equation}
We will use the following notation throughout the text. For $f:\SigmaNS\to\R$, let $\langle f\rangle$ denote the quenched Gibbs average of $f$, i.e. the average with the disorder $(g_{ij})_{i<j}$ fixed. Let $\E$ denote the expectation with respect to the disorder and some other auxiliary Gaussian randomness.

 For different replicas $\sigma^1, \sigma^2 \in \gS_{N,S}$, define the overlap
\begin{equation}\label{eq:overlap-def}
  R_{1,2}:=\frac1N\sum_{i=1}^N\sigma_i^1\sigma_i^2.
\end{equation}
When $\sigma^1=\sigma^2$, the self-overlap is 
\begin{equation}\label{eq:self-overlap-def}
  R_{1,1}:=\frac1N\sum_{i=1}^N(\sigma_i^1)^2.
\end{equation}
 In the Ising SK model, the self overlap reduces to 1, but in the GS model it is configuration dependent and must be handled together with the usual overlap.

We now fix the high-temperature notation used throughout the paper. Let $Z$ be a standard Gaussian random variable. In the high-temperature region of Auffinger-Chen \cite{AC21} and Sheng-Wu \cite{SW24}, let $(p,q)$ denote the unique solution of the self-consistent equations
\begin{align}
  p &= \E_Z\left[\frac{\sum_{m=-S}^{S}m^2\exp(\Dtil m^2+m(\beta\sqrt q Z+h))}
                 {\sum_{m=-S}^{S}\exp(\Dtil m^2+m(\beta\sqrt q Z+h))}\right], \label{eq:p-equation}\\
  q &= \E_Z\left[\left(\frac{\sum_{m=-S}^{S}m\exp(\Dtil m^2+m(\beta\sqrt q Z+h))}
                 {\sum_{m=-S}^{S}\exp(\Dtil m^2+m(\beta\sqrt q Z+h))}\right)^2\right], \label{eq:q-equation}
\end{align}
where
\begin{equation}\label{eq:v-Dtilde}
  v:=p-q, \qquad \Dtil:=D+\frac{\beta^2}{2}(p-q), \qquad Z \sim \cN(0,1).
\end{equation}
Here and below, $\beta_0>0$ denotes a high-temperature constant small enough so that, uniformly for $\beta\le \beta_0$, the uniqueness and local Lipschitz dependence of $(p,q)$ from \cite{AC21} and the quantitative overlap and self-overlap moment estimates from \cite{SW24} hold. In particular, for every fixed $k\ge1$,
\begin{equation}\label{eq:known-overlap-estimate}
  \E\langle |R_{1,2}-q|^{2k}\rangle \le \frac{C}{N^k}, \qquad \E\langle |R_{1,1}-p|^{2k}\rangle\le \frac{C}{N^k}.
\end{equation}
We write
\begin{equation}\label{eq:W-defs}
  W(x):=\sum_{m=-S}^{S}\exp(Dm^2+mx),\qquad
  \Wtil(x):=\sum_{m=-S}^{S}\exp(\Dtil m^2+mx).
\end{equation}
Finally, for $a\in\R$ and $s>0$, $\varphi_{a,s}$ denotes the Gaussian density with mean $a$ and variance $s$.

We also record the function class used throughout the interpolation arguments. A function $U:\R^d\to\R$ is said to have moderate growth if, for every $a>0$,
\begin{equation}\label{eq:moderate-growth}
  \sup_{x\in\R^d}|U(x)|\exp(-a\|x\|^2)<\infty.
\end{equation}
When we say that $U$ is smooth with derivatives of moderate growth, we mean that each partial derivative of $U$ satisfies \eqref{eq:moderate-growth}. This is the same convention used in Chen's paper \cite{Chen2013}.

\subsection{Main results}

We are interested in the behavior of the limiting distribution of the local fields. Before that, we state a central limit theorem for the cavity fields of the GS model, which will be used in proving the law of local fields. Let $J_1,\ldots,J_N$ be i.i.d. standard Gaussian variables independent of the disorder in $H_N$, and define the cavity field and its Gibbs average,
\begin{equation}\label{eq:cavity-field-intro}
  \ell=\frac1{\sqrt N}\sum_{j=1}^NJ_j\sigma_j,
  \qquad
  r=\langle \ell\rangle=\frac1{\sqrt N}\sum_{j=1}^NJ_j\langle\sigma_j\rangle.
\end{equation}

\begin{theorem}[CLT of cavity field]\label{thm:cavity}
Let $\beta_0$ be as above and $k\in\N$. Suppose that $U$ is an infinitely differentiable function on $\R$ whose derivatives of all orders are of moderate growth. Then for $\beta\le \beta_0$ and $D,h\in\R$, we have
\begin{equation}\label{eq:thm-cavity}
  \E\left(\langle U(\ell)\rangle-\int_{\R}U(x)\varphi_{r,p-q}(x)d x\right)^{2k}\le \frac{C}{N^k},
\end{equation}
where $C$ is a constant depending on $k,U,S,D,h$, and the high-temperature bound only. The second subscript in $\varphi_{u,v}$ is the variance, as in the convention above.
\end{theorem}

For the last site $N$, write the \emph{local field} as
\[
  l_N=\frac1{\sqrt N}\sum_{j<N}g_{jN}\sigma_j.
\]
Now we introduce the limiting law of the local fields, which is a finite mixture of Gaussian distributions:
\begin{equation}\label{eq:nu-def}
  d\nu_N(x):=\sum_{m=-S}^{S}p_m\varphi_{\gamma_N+m\beta(p-q),\,p-q}(x)d x,
\end{equation}
where
\begin{equation}\label{eq:weights-gamma}
  p_m:=\frac{\exp(\Dtil m^2+m(\beta\gamma_N+h))}{\Wtil(\beta\gamma_N+h)},
  \qquad
  \gamma_N:=\langle l_N\rangle-\beta(p-q)\langle\sigma_N\rangle.
\end{equation}

\begin{theorem}[Limiting law of local field]\label{thm:local}
Under the same assumptions as in Theorem~\ref{thm:cavity}, we have
\begin{equation}\label{eq:thm-local}
  \E\left(\langle U(l_N)\rangle-\int_{\R}U(x)d\nu_N(x)\right)^{2k}\le \frac{C}{N^k},
\end{equation}
where $C$ is a finite constant depending on $k,U,S,D,h$, and the high-temperature bound only.
\end{theorem}

The measure $\nu_N$ is random through $\gamma_N$. In contrast to the results in SK model by~\cite{Chen2013,Chatterjee2010}, where the limiting law of local field law is mixture of two Gaussian with symmetric mean, we now have a more involved structure.

\subsection{Structure of the paper}

Section~\ref{sec:cavity} proves the cavity-field central limit theorem. In Section~\ref{sec:local}, we apply the cavity decomposition to the last coordinate, then prove the Onsager correction estimate and identify the Gaussian-mixture limiting law.

\subsection{Acknowledgement}
We would like to express our sincere gratitude to our mentor, Dr. Qiang Wu, for his generous guidance, thoughtful suggestions, and kind support throughout the entire process of this project since fall 2025. 

\section{Proof of Theorem~\ref{thm:cavity}}\label{sec:cavity}

In this section, we first establish concentration bounds for overlap fluctuations in the GS model. These quantities will appear naturally in the Gaussian interpolation procedure. We then introduce an interpolation connecting the GS cavity field to its Gaussian limit and derive derivative formulas for the corresponding interpolating functional. Finally, a recursive integral expansion and endpoint cancellation allow us to prove the central limit theorem.

Let $J_1,\ldots,J_N$ be i.i.d. standard Gaussian variables, independent of the disorder in the $N$-spin Hamiltonian. Using replicas, for $1\le i\le 2k$, set
\begin{equation}\label{eq:ell-i}
  \ell^i=\frac1{\sqrt N}\sum_{j=1}^N J_j\sigma_j^i.
\end{equation}
Set
\begin{equation}\label{eq:r-i}
  r=\langle\ell^i\rangle=\frac1{\sqrt N}\sum_{j=1}^NJ_j\langle\sigma_j\rangle,
\end{equation}
so that
\begin{equation}\label{eq:ell-centered}
  \ell^i-r=\frac1{\sqrt N}\sum_{j=1}^NJ_j(\sigma_j^i-\langle\sigma_j\rangle).
\end{equation}
Since $\ell^i-r$ is a Gaussian linear combination of centered spins, we expect Gaussian fluctuations. To identify the variance, consider $N^{-1}\sum_{j=1}^N(\sigma_j-\langle\sigma_j\rangle)^2$. Expanding and taking expectations gives $\E\langle R_{1,1}\rangle-\E\langle R_{1,2}\rangle=p-q+O(N^{-1/2})$ by \eqref{eq:known-overlap-estimate}, suggesting that $\ell^i-r$ can be approximated by a Gaussian random variable with variance $p-q$.

This motivates approximating $\ell^i-r$ by a Gaussian random variable with variance $p-q$. Let $\xi,\xi^1,\ldots,\xi^{2k}$ be i.i.d. centered Gaussian random variables with variance $p-q$, independent of both $(J_j)_{j\le N}$ and the disorder $(g_{ij})_{i<j\le N}$. For $t\in[0,1]$, define
\begin{equation}\label{eq:ui-def}
  u_i(t)=\sqrt t(\ell^i-r)+\sqrt{1-t}\,\xi^i.
\end{equation}
Suppose that $U$ is a real-valued function defined on $\R$ of moderate growth. Define
\begin{equation}\label{eq:V-def}
  V(x,y)=U(x+y)-\E_{\xi}[U(\xi+y)]
\end{equation}
and
\begin{equation}\label{eq:psi-def}
  \psi(t)=\E\left\langle\prod_{i=1}^{2k}V(u_i(t),r)\right\rangle.
\end{equation}

To study the interpolation function $\psi(t)$, we need to analyze the terms appearing after taking derivatives of $\psi(t)$. These terms involve centered overlap fluctuations. Define centered spins by
\begin{equation}\label{eq:sigma-dot}
  \dot\sigma_j^i=\sigma_j^i-\langle\sigma_j\rangle.
\end{equation}
The fluctuation terms arising from the interpolation can be expressed through the following quantities:
\begin{equation}\label{eq:T-defs}
  T_i=\frac1N\sum_{m\le N}\dot\sigma_m^i\langle\sigma_m\rangle,
  \qquad
  T_{i,i}=\frac1N\sum_{m\le N}(\dot\sigma_m^i)^2-(p-q),
  \qquad
  T_{i,j}=\frac1N\sum_{m\le N}\dot\sigma_m^i\dot\sigma_m^j,\quad i\ne j.
\end{equation}

\begin{lemma}\label{lem:T-estimate}
There exists $\beta_0>0$ such that for $\beta\le\beta_0$ we have
\begin{equation}\label{eq:T-estimate}
  \max_{1\le i,j\le 2k,\,i\ne j}\Big\{\E\langle |T_i|^{2k}\rangle,
  \E\langle |T_{i,i}|^{2k}\rangle,
  \E\langle |T_{i,j}|^{2k}\rangle\Big\}\le \frac{C}{N^k},
\end{equation}
where $C$ is a constant independent of $N$.
\end{lemma}

\begin{proof}
For replicas $\sigma^1,\sigma^2,\ldots$, use the overlap notation \eqref{eq:overlap-def}. We use $\langle\cdot\rangle_a$ to denote the Gibbs average with respect to the replica $\sigma^a$, and $\langle\cdot\rangle_{a,b}$ for the Gibbs average with respect to both $\sigma^a$ and $\sigma^b$. For notational convenience, we use replicas $1$ and $2$ as auxiliary replicas in the following computations, and assume that they are distinct from the fixed replicas $i,j$. This is harmless by exchangeability of replicas. For example,
\[
  \langle R_{i,2}\rangle_2=\frac1N\sum_{m=1}^N\sigma_m^i\langle\sigma_m\rangle,
  \qquad
  \langle R_{1,2}\rangle_{1,2}=\frac1N\sum_{m=1}^N\langle\sigma_m\rangle^2.
\]
With this notation, expanding $T_i$ gives
\begin{align}
  T_i&=\frac1N\sum_{m=1}^N(\sigma_m^i-\langle\sigma_m\rangle)\langle\sigma_m\rangle \notag\\
     &=\frac1N\sum_{m=1}^N\sigma_m^i\langle\sigma_m\rangle-\frac1N\sum_{m=1}^N\langle\sigma_m\rangle^2 \notag\\
     &=\langle R_{i,2}\rangle_2-\langle R_{1,2}\rangle_{1,2}
      =\langle R_{i,2}-q\rangle_2-\langle R_{1,2}-q\rangle_{1,2}. \label{eq:Ti-expansion}
\end{align}
Similarly, for $T_{i,i}$, expanding the square gives
\begin{align}
  T_{i,i}&=\frac1N\sum_{m=1}^N(\sigma_m^i-\langle\sigma_m\rangle)^2-(p-q) \notag\\
  &=\frac1N\sum_{m=1}^N\big((\sigma_m^i)^2-2\sigma_m^i\langle\sigma_m\rangle+\langle\sigma_m\rangle^2\big)-(p-q) \notag\\
  &=R_{i,i}-2\langle R_{i,2}\rangle_2+\langle R_{1,2}\rangle_{1,2}-(p-q) \notag\\
  &=(R_{i,i}-p)-2\langle R_{i,2}-q\rangle_2+\langle R_{1,2}-q\rangle_{1,2}. \label{eq:Tii-expansion}
\end{align}
Finally, for $i\ne j$, expanding the product gives
\begin{align}
  T_{i,j}&=\frac1N\sum_{m=1}^N(\sigma_m^i-\langle\sigma_m\rangle)(\sigma_m^j-\langle\sigma_m\rangle) \notag\\
  &=R_{i,j}-\langle R_{i,2}\rangle_2-\langle R_{j,2}\rangle_2+\langle R_{1,2}\rangle_{1,2} \notag\\
  &=(R_{i,j}-q)-\langle R_{i,2}-q\rangle_2-\langle R_{j,2}-q\rangle_2+\langle R_{1,2}-q\rangle_{1,2}. \label{eq:Tij-expansion}
\end{align}
Thus $T_i$ and $T_{i,j}$ are controlled by overlap fluctuations of the form $R_{a,b}-q$, while $T_{i,i}$ also involves the self-overlap fluctuation $R_{a,a}-p$.

By the overlap and self-overlap concentration results established in Sheng-Wu \cite{SW24}, there exists $\beta_0>0$ such that for $\beta\le\beta_0$,
\begin{equation}\label{eq:overlap-concentration}
  \E\langle |R_{1,2}-q|^{2k}\rangle+\E\langle |R_{1,1}-p|^{2k}\rangle\le \frac{C}{N^k}.
\end{equation}
We use Jensen's inequality to remove the additional Gibbs averages. For example,
\[
  \E\left\langle |\langle R_{i,2}-q\rangle_2|^{2k}\right\rangle
  \le \E\left\langle |R_{i,2}-q|^{2k}\right\rangle.
\]
The same argument applies to the terms averaged over replicas $1$ and $2$. Together with the elementary inequality
\[
  |x_1+\cdots+x_m|^{2k}\le m^{2k-1}\sum_{\ell=1}^m|x_\ell|^{2k},
\]
the identities \eqref{eq:Ti-expansion}, \eqref{eq:Tii-expansion}, and \eqref{eq:Tij-expansion}, together with \eqref{eq:overlap-concentration}, imply
\[
  \E\langle |T_i|^{2k}\rangle+\E\langle |T_{i,i}|^{2k}\rangle+\E\langle |T_{i,j}|^{2k}\rangle\le \frac{C}{N^k}.
\]
Taking the maximum over $i,j$ gives the desired bound.
\end{proof}

Lemma~\ref{lem:T-estimate} gives quantitative bounds for the fluctuation terms introduced above. We also need control on functions involving Gaussian random variables and moderate growth functions. The following form of Chen's integrability lemma will be used repeatedly in later arguments.

\begin{lemma}[ \cite{Chen2013}*{Lemma 2.3}]\label{lem:chen-integrability}
Let $I\subset[0,\infty)$ and $J$ be two bounded intervals. Let $g_1,\ldots,g_N$ be independent standard Gaussian variables, and let $\xi\sim N(0,v)$ be independent of them. Let $K_1,K_2,K_3>0$ be constants, and let $\mathcal F$ be a sigma-field independent of the Gaussian family $((g_j)_{j\le N},\xi)$. Assume the following conditions.
\begin{enumerate}[label=(\roman*),leftmargin=2.4em]
\item For each $N\in\N$ and $\beta\in J$, the random variables $\{X^N_{j,\beta}:1\le j\le N\}$ are $\mathcal F$-measurable and satisfy
\[
  |X^N_{j,\beta}|\le K_1\qquad\text{for all }1\le j\le N,\ \beta\in J.
\]
\item The measurable functions $f_1,f_2:\N\times I\times J\to\R$ satisfy, for all $(N,t,\beta)\in\N\times I\times J$,
\[
  |f_1(N,t,\beta)|\le \frac{K_2}{\sqrt N},
  \qquad
  |f_2(N,t,\beta)|\le K_3.
\]
\item The continuous function $U:\R\to\R$ satisfies, for some $A>0$ and
\[
  0<a<\min\left\{\frac1{8K_1^2K_2^2},\frac1{4vK_3^2}\right\},
\]
that
\[
  |U(x)|\le A e^{a x^2}\qquad\text{for all }x\in\R.
\]
\end{enumerate}
Then there exists a constant $K>0$, depending only on $(A,a,K_1,K_2,K_3,v,I,J)$, such that
\begin{equation}\label{eq:chen-integrability}
  \sup_{N\in\N}\sup_{\beta\in J}
  \E_0\left[\sup_{t\in I}\left|U\left(f_1(N,t,\beta)\sum_{j=1}^Ng_jX^N_{j,\beta}+f_2(N,t,\beta)\xi\right)\right|\right]\le K,
\end{equation}
where $\E_0$ denotes conditional expectation with respect to the Gaussian randomness $((g_j)_{j\le N},\xi)$, with $\mathcal F$ fixed. In particular, the bound is uniform in the realization of the additional randomness.
\end{lemma}

To study the interpolation function, we next derive a conditional derivative formula through Gaussian integration by parts, with the spins fixed. After applying the formula, we average over the product Gibbs measure and over the disorder. The derivative can be expressed in terms of the fluctuation quantities introduced above.

\begin{lemma}[\cite{Chen2013}*{Lemma 2.4}]\label{lem:derivative}
Let $k\in\N$. Suppose that $V_1,\ldots,V_{2k}:\R^2\to\R$ are twice continuously differentiable and of moderate growth, together with their first and second partial derivatives. Define
\[
  \phi(t)=\E_0\left[\prod_{i=1}^{2k}V_i(u_i(t),r)\right],\qquad 0\le t\le 1.
\]
Then $\phi$ is differentiable on $(0,1)$ and
\begin{align}\label{eq:derivative-formula}
\phi'(t)
=&\frac12\sum_{i\le 2k}T_{i,i}\E_0\left[\frac{\partial^2V_i}{\partial x^2}(u_i(t),r)\prod_{\substack{j\le 2k\\ j\ne i}}V_j(u_j(t),r)\right] \notag\\
&+\frac12\sum_{\substack{i,j\le 2k\\ i\ne j}}T_{i,j}\E_0\left[\frac{\partial V_i}{\partial x}(u_i(t),r)\frac{\partial V_j}{\partial x}(u_j(t),r)
\prod_{\substack{h\le 2k\\ h\ne i,j}}V_h(u_h(t),r)\right] \notag\\
&+\frac1{2\sqrt t}\sum_{i\le 2k}T_i\E_0\left[\frac{\partial^2V_i}{\partial x\partial y}(u_i(t),r)\prod_{\substack{j\le 2k\\ j\ne i}}V_j(u_j(t),r)\right] \notag\\
&+\frac1{2\sqrt t}\sum_{\substack{i,j\le 2k\\ i\ne j}}T_i\E_0\left[\frac{\partial V_i}{\partial x}(u_i(t),r)\frac{\partial V_j}{\partial y}(u_j(t),r)
\prod_{\substack{h\le 2k\\ h\ne i,j}}V_h(u_h(t),r)\right].
\end{align}
\end{lemma}

\begin{remark}
The moment bounds in Lemma~\ref{lem:T-estimate} will be repeatedly used in controlling the derivatives of the interpolation function.
\end{remark}

Lemma~\ref{lem:derivative} will be applied recursively. The singular factors $t^{-1/2}$ prevent a direct Taylor expansion at $t=0$. We therefore use the following integral expansion, which is the form of Chen's Lemma~2.7 needed here.

\begin{lemma}[\cite{Chen2013}*{Lemma~2.7}]\label{lem:recursive-expansion}
Fix an integer $m\ge1$. Let $\psi$ be continuous on $[0,1]$ and differentiable on $(0,1)$. For a binary word $\mathbf s=(s_1,\ldots,s_n)\in\{0,1\}^n$, let $\psi_{\mathbf s}$ be continuous on $[0,1]$ and differentiable on $(0,1)$. Suppose that all the expressions below are integrable at zero, that $\psi(0)=0$, and that
\begin{equation}\label{eq:recursive-derivatives}
  \psi'(t)=\frac12\psi_{(0)}(t)+\frac1{2\sqrt t}\psi_{(1)}(t),
  \qquad
  \psi_{\mathbf s}'(t)=\frac12\psi_{(\mathbf s,0)}(t)
  +\frac1{2\sqrt t}\psi_{(\mathbf s,1)}(t)
\end{equation}
for $t\in(0,1)$ and $1\le |\mathbf s|\le m$. If
\[
  \psi_{\mathbf s}(0)=0,\qquad 1\le |\mathbf s|<m,
\]
then, writing
\[
  \Delta_n(t):=\{(t_1,\ldots,t_n):0<t_n<\cdots<t_1<t\},
\]
we have
\begin{align}
\psi(t)
={}&\frac1{2^m}\sum_{\mathbf s\in\{0,1\}^m}
\int_{\Delta_m(t)}
\left(\prod_{j=1}^m t_j^{-s_j/2}\right)\psi_{\mathbf s}(0)\,d\mathbf t \label{eq:recursive-expansion}\\
&+\frac1{2^{m+1}}\sum_{\mathbf s\in\{0,1\}^{m+1}}
\int_{\Delta_{m+1}(t)}
\left(\prod_{j=1}^{m+1}t_j^{-s_j/2}\right)
\psi_{\mathbf s}(t_{m+1})\,d\mathbf t.\notag
\end{align}
\end{lemma}

\begin{proof}
Integrate the first identity in \eqref{eq:recursive-derivatives} from $\varepsilon$ to $t$ and then let $\varepsilon\downarrow0$. Continuity at zero and integrability of the right-hand side give the assertion after one step. At each of the next $m-1$ steps, the endpoint term vanishes by the assumption $\psi_{\mathbf s}(0)=0$ for $|\mathbf s|<m$. Iterating the second identity therefore gives, by induction,
\[
  \psi(t)=\frac1{2^m}\sum_{\mathbf s\in\{0,1\}^m}
  \int_{\Delta_m(t)}\left(\prod_{j=1}^m t_j^{-s_j/2}\right)
  \psi_{\mathbf s}(t_m)\,d\mathbf t.
\]
For each word of length $m$, write
\[
  \psi_{\mathbf s}(t_m)=\psi_{\mathbf s}(0)
  +\int_0^{t_m}\psi_{\mathbf s}'(t_{m+1})\,dt_{m+1}
\]
and use \eqref{eq:recursive-derivatives} once more. This gives \eqref{eq:recursive-expansion}. The factors $t_j^{-1/2}$ are integrable at zero, so all the displayed integrals are finite under the stated assumptions. The identity at $t=1$ follows by continuity.
\end{proof}

We are now ready to prove Theorem~\ref{thm:cavity}.

\begin{proof}[Proof of Theorem~\ref{thm:cavity}]
Recall the definitions of $V$ and $\psi$. By Lemma~\ref{lem:derivative}, after averaging over the replicas and the disorder, there are functions $\psi_{(0)}$ and $\psi_{(1)}$ such that
\[
  \psi'(t)=\frac12\psi_{(0)}(t)+\frac1{2\sqrt t}\psi_{(1)}(t).
\]
Applying the same derivative formula to every term that occurs, and continuing recursively, defines functions $\psi_{\mathbf s}$ for binary words $\mathbf s$ of length at most $2k+1$ satisfying \eqref{eq:recursive-derivatives}. Differentiation under the outer $\E\langle\cdot\rangle$ is justified by the same domination argument as in \cite{Chen2013}*{Lemma~2.4}, using Lemma~\ref{lem:chen-integrability}. Lemma~\ref{lem:chen-integrability} and the moderate-growth assumptions imply that these functions have continuous extensions to $t=0$ and that the singular terms are integrable there. Moreover, $\psi(0)=0$, since at $t=0$ the variables $\xi^1,\ldots,\xi^{2k}$ are independent and each factor $V(\xi^i,r)$ has zero $\xi^i$-mean.

We record the structure of the recursively generated terms. If $|\mathbf s|=n$, every summand in $\psi_{\mathbf s}(t)$ is a constant multiple of
\begin{equation}\label{eq:typical-recursive-term}
  \E\left\langle
  \left(\prod_{i\le2k}T_{i,i}^{a_i}\right)
  \Bigg(\prod_{\substack{i,j\le2k\\i\ne j}}T_{i,j}^{b_{ij}}\Bigg)
  \left(\prod_{i\le2k}T_i^{c_i}\right)
  \E_0\left[\prod_{i\le2k}
  \frac{\partial^{d_i+e_i}V}{\partial x^{d_i}\partial y^{e_i}}(u_i(t),r)
  \right]\right\rangle,
\end{equation}
where the nonnegative integers satisfy
\begin{align}
  & \sum_i a_i+\sum_{i\ne j}b_{ij}+\sum_i c_i=n, \label{eq:recursive-count-1}\\
  & d_i=2a_i+\sum_{j\ne i}(b_{ij}+b_{ji})+c_i, \qquad i\le2k, \label{eq:recursive-count-2}\\
  & \sum_i c_i=\sum_{j=1}^n s_j=\sum_i e_i. \label{eq:recursive-count-3}
\end{align}
Indeed, a regular term in Lemma~\ref{lem:derivative} adds either one $T_{i,i}$ and two $x$-derivatives, or one $T_{i,j}$ and one $x$-derivative at each of $i,j$. A singular term adds one $T_i$, one $x$-derivative at $i$, and one $y$-derivative. Thus \eqref{eq:recursive-count-1}-\eqref{eq:recursive-count-3} follow by induction.

We next prove the endpoint cancellation. At $t=0$, $u_i(0)=\xi^i$, and the variables $\xi^1,\ldots,\xi^{2k}$ are independent. If $d_i=0$ for some $i$, then
\[
  \E_{\xi^i}\left[\frac{\partial^{e_i}V}{\partial y^{e_i}}(\xi^i,r)\right]=0
\]
by the definition of $V$, so the term \eqref{eq:typical-recursive-term} vanishes. On the other hand, if $d_i=1$, then no factor $T_{i,i}$ is present and the label $i$ occurs in exactly one factor of type $T_i$, $T_{i,j}$, or $T_{j,i}$. Expanding that factor leaves a single centered spin with replica label $i$, so averaging over replica $i$ makes the term vanish because $\langle\dot\sigma_j^i\rangle_i=0$. Consequently, a nonzero endpoint term must satisfy $d_i\ge2$ for every $i\le2k$. From \eqref{eq:recursive-count-1} and \eqref{eq:recursive-count-2},
\[
  2n=\sum_{i\le2k}d_i+\sum_{i\le2k}c_i\ge4k.
\]
It follows that
\begin{equation}\label{eq:endpoint-cancellation}
  \psi_{\mathbf s}(0)=0,\qquad |\mathbf s|<2k.
\end{equation}

If $n=2k$ and an endpoint term is nonzero, equality must hold throughout the preceding display. Hence $d_i=2$ for every $i$, all $c_i$ and $e_i$ vanish, and $\mathbf s=(0,\ldots,0)$. Thus every nonzero term in $\psi_{(0,\ldots,0)}(0)$ contains exactly $2k$ factors from the family $T_i,T_{i,i},T_{i,j}$, and its Gaussian factor is a product of $U''(\xi^i+r)$. More generally, moderate growth bounds every Gaussian factor at the finite recursion depths used here by $C_a\exp\{a\sum_i(u_i(t)^2+r^2)\}$, Lemma~\ref{lem:chen-integrability} and H\"older's inequality give a uniform conditional bound. A further application of H\"older's inequality and Lemma~\ref{lem:T-estimate} gives
\begin{equation}\label{eq:leading-recursive-bound}
  |\psi_{(0,\ldots,0)}(0)|\le \frac{C}{N^k}.
\end{equation}
Likewise, every summand in a function $\psi_{\mathbf s}$ with $|\mathbf s|=2k+1$ contains $2k+1$ fluctuation factors. Applying Lemma~\ref{lem:T-estimate} with moment order $4k+2$ and H\"older's inequality yields
\begin{equation}\label{eq:remainder-recursive-bound}
  \sup_{t\in[0,1]}\max_{|\mathbf s|=2k+1}|\psi_{\mathbf s}(t)|
  \le \frac{C}{N^{k+1/2}}.
\end{equation}
The constants in these estimates also include the uniform Gaussian bounds from Lemma~\ref{lem:chen-integrability}, only finitely many derivatives of $U$ occur for fixed $k$.

Apply Lemma~\ref{lem:recursive-expansion} with $m=2k$ and $t=1$. By \eqref{eq:endpoint-cancellation}, the first part of \eqref{eq:recursive-expansion} has only the all-zero word, whose simplex integral equals $1/(2k)!$. Every kernel in the remainder is integrable, since each exponent is either $0$ or $-1/2$. Therefore \eqref{eq:leading-recursive-bound} and \eqref{eq:remainder-recursive-bound} give
\[
  |\psi(1)|\le \frac{C}{N^k}.
\]
Finally, at $t=1$, $u_i(1)=\ell^i-r$, and independence of the replicas gives
\[
  \psi(1)=\E\big(\langle U(\ell)\rangle-\E_\xi[U(\xi+r)]\big)^{2k}.
\]
Since $\xi+r$ is Gaussian with mean $r$ and variance $p-q$,
\[
  \E_\xi[U(\xi+r)]=\int_\R U(x)\varphi_{r,p-q}(x)\,dx.
\]
This proves Theorem~\ref{thm:cavity}.
\end{proof}

\section{Proof of Theorem~\ref{thm:local}}\label{sec:local}

Before presenting the proof details of Theorem~\ref{thm:local}, we need to set up some necessary notations. Recall $\beta_0$ obtained from the overlap concentration results in Lemma~\ref{lem:T-estimate}. In the cavity approach, we decompose the original Hamiltonian as follows:
\[
  H_N(\sigma)=H_{N-1}(\sigma^-)+\beta\sigma_N l_N+D\sigma_N^2+h\sigma_N,
\]
where
\begin{equation}\label{eq:HNminus}
  H_{N-1}(\sigma^-):=\frac{\beta}{\sqrt N}\sum_{i<j<N}g_{ij}\sigma_i\sigma_j+D\sum_{i=1}^{N-1}\sigma_i^2+h\sum_{i=1}^{N-1}\sigma_i,
  \qquad
  l_N:=\frac1{\sqrt N}\sum_{j<N}g_{jN}\sigma_j.
\end{equation}
For convenience we will use the following notation from now on:
\begin{equation}\label{eq:minus-notation}
  \gamma=\gamma_N,\qquad l=l_N,
  \qquad
  l_-:=\frac1{\sqrt{N-1}}\sum_{j<N}g_{jN}\sigma_j=\sqrt{\frac{N}{N-1}}\,l,
  \qquad
  r_-:=\langle l_-\rangle_-.
\end{equation}
We also define
\begin{equation}\label{eq:beta-minus}
  \beta_-:=\beta\sqrt{\frac{N-1}{N}},
\end{equation}
and note that $\beta_-l_-=\beta l$. The Gibbs average $\langle\cdot\rangle_-$ is with respect to $H_{N-1}(\sigma^-)$.

We will use the following notations, equivalent to the hyperbolic-cosine form by symmetry of the spin set:
\begin{equation}\label{eq:W-section3}
  W(x):=\sum_{m=-S}^{S}e^{Dm^2+mx},
  \qquad
  \Dtil:=D+\frac{\beta^2}{2}(p-q),
  \qquad
  \Wtil(x):=\sum_{m=-S}^{S}e^{\Dtil m^2+mx}.
\end{equation}
While applying the cavity method, for the size $N-1$ system, let
\begin{equation}\label{eq:minus-params}
  \Dtil_-:=D+\frac{\beta_-^2}{2}(p_- -q_-),
  \qquad
  \Wtil_-(x):=\sum_{m=-S}^{S}e^{\Dtil_-m^2+mx}.
\end{equation}
Here $p_-,q_-$ are the solution to the fixed point equations obtained by replacing $\beta$ by $\beta_-$ in \eqref{eq:p-equation} and \eqref{eq:q-equation}. Set $v_-=p_- -q_-$. In the intermediate estimates below, $\E_{\xi_-}$ denotes expectation with respect to a centered Gaussian variable $\xi_-$ of variance $v_-$. To keep notation light, some displays write $\E_\xi$ in place of $\E_{\xi_-}$ until the final variance comparison step.

\begin{lemma}\label{lem:param-comparison}
There exists a positive constant $C$ independent of $N$ such that
\begin{equation}\label{eq:param-comparison}
  |\beta-\beta_-|,  \ |p-p_-|, \ |q-q_-|, \ |\Dtil-\Dtil_-|, \ |v-v_-|\le \frac{C}{N}.
\end{equation}
\end{lemma}

\begin{proof}
The estimate for $\beta$ follows from
\[
  \beta-\beta_- = \beta\left(1-\sqrt{1-\frac1N}\right)=O(N^{-1}).
\]
In the high-temperature region, the fixed point $(p(\beta),q(\beta))$ is unique and locally Lipschitz in $\beta$ by the implicit-function argument used in Auffinger-Chen \cite{AC21}. Therefore
\[
  |p-p_-|+|q-q_-|\le C|\beta-\beta_-|\le \frac{C}{N}.
\]
The estimate for $v-v_-$ follows immediately. Finally,
\[
  |\Dtil-\Dtil_-|=\left|\frac{\beta^2}{2}(p-q)-\frac{\beta_-^2}{2}(p_- -q_-)\right|\le \frac{C}{N},
\]
using boundedness of $p,q,p_-,q_-$ by $S^2$.
\end{proof}

\begin{proposition}\label{prop:ratio}
Let $\beta\le \beta_0$ and $h,D\in\R$. Let $U:\R\to\R$ be smooth, with derivatives of moderate growth. For a cavity field $l$ and $r=\langle l\rangle$, with Gaussian reference variance $p-q$, we have, for every positive integer $k$,
\begin{equation}\label{eq:ratio-estimate}
  \E\left[\frac{\langle U(l)W(\beta l+h)\rangle}{\langle W(\beta l+h)\rangle}
  -\frac{\E_\xi[U(\xi+r)W(\beta(\xi+r)+h)]}{\Wtil(\beta r+h)}\right]^{2k}\le \frac{K}{N^k}.
\end{equation}
Here $K$ is uniform over the system size and over $0<\beta\le\beta_0$, by the uniformity of \eqref{eq:known-overlap-estimate}.
\end{proposition}

\begin{proof}
Define for $m\in\{-S,\ldots,S\}$
\[
  A(m)=\langle U(l)e^{m\beta l}\rangle-\E_\xi[U(\xi+r)e^{m\beta(\xi+r)}],
\]
\[
  B(m)=\langle e^{m\beta l}\rangle-\E_\xi[e^{m\beta(\xi+r)}].
\]
By Theorem~\ref{thm:cavity} applied to $F(x)=U(x)e^{m\beta x}$ and $F(x)=e^{m\beta x}$ respectively, for each $m$ we have
\begin{equation}\label{eq:Am-Bm}
  \E|A(m)|^{4k}\le \frac{K}{N^{2k}},
  \qquad
  \E|B(m)|^{8k}\le \frac{K}{N^{4k}}.
\end{equation}
Define
\[
  A=\langle U(l)W(\beta l+h)\rangle,
  \qquad
  B=\langle W(\beta l+h)\rangle,
\]
\[
  A'=\E_\xi[U(\xi+r)W(\beta(\xi+r)+h)],
  \qquad
  B'=\Wtil(\beta r+h),
\]
and note
\[
  A-A'=\sum_{m=-S}^{S}e^{mh+Dm^2}A(m),
  \qquad
  B-B'=\sum_{m=-S}^{S}e^{mh+Dm^2}B(m).
\]
We can see
\begin{align*}
  \E\left[\frac{A}{B}-\frac{A'}{B'}\right]^{2k}
  &=\E\left[\frac{A-A'}{B}+\frac{A'}{B}\cdot\frac{B'-B}{B'}\right]^{2k}\\
  &\le 2^{2k}\E\left[\left(\frac{A-A'}{B}\right)^{2k}
  +\left(\frac{A'}{B}\right)^{2k}\left(\frac{B'-B}{B'}\right)^{2k}\right].
\end{align*}
We need lower bounds on $B$ and $B'$. Since every term in $W$ is strictly positive and the $m=0$ term contributes $1$,
\[
  B=\langle W(\beta l+h)\rangle\ge1,
  \qquad
  B'=\Wtil(\beta r+h)\ge1.
\]
Since all terms in the sum defining $B$ are positive, for each fixed $m\in\{-S,\ldots,S\}$,
\[
  B\ge e^{mh+Dm^2}\langle e^{m\beta l}\rangle,
  \qquad
  B'\ge e^{mh+Dm^2}\E_\xi[e^{m\beta(\xi+r)}].
\]
By Jensen's inequality,
\[
  \langle e^{m\beta l}\rangle\ge e^{m\beta\langle l\rangle}=e^{m\beta r},
  \qquad
  \E_\xi[e^{m\beta(\xi+r)}]\ge e^{m\beta r}.
\]
Thus $1/\langle e^{m\beta l}\rangle\le e^{-m\beta r}$ and $1/\E_\xi[e^{m\beta(\xi+r)}]\le e^{-m\beta r}$. We can see
\[
  \left|\frac{A-A'}{B}\right|
  \le \sum_{m=-S}^{S}\frac{e^{mh+Dm^2}|A(m)|}{B}
  \le \sum_{m=-S}^{S}|A(m)|e^{-m\beta r}.
\]
It suffices to bound each term. Using \eqref{eq:Am-Bm} and Chen's integrability estimate, Lemma~\ref{lem:chen-integrability} (that is, \cite[Lemma~2.3]{Chen2013}), since $r=N^{-1/2}\sum_j J_j\langle\sigma_j\rangle$ with $|\langle\sigma_j\rangle|\le S$ and $|m|\beta\le S\beta_0$, we obtain
\[
  \E\big[|A(m)|e^{-m\beta r}\big]^{2k}
  \le \big(\E|A(m)|^{4k}\big)^{1/2}\big(\E e^{-4km\beta r}\big)^{1/2}
  \le \frac{K}{N^k}.
\]
Therefore, summing over $m$ gives
\begin{equation}\label{eq:A-over-B-bound}
  \E\left|\frac{A-A'}{B}\right|^{2k}\le \frac{K}{N^k}.
\end{equation}
Now, we continue with
\[
  \left|\frac{A'}{B}\right|
  \le \sum_{m=-S}^{S}\E_\xi[|U(\xi+r)|e^{m\beta(\xi+r)}]e^{-m\beta r}
  =\sum_{m=-S}^{S}\E_\xi[|U(\xi+r)|e^{m\beta\xi}].
\]
For each term, we again use Lemma~\ref{lem:chen-integrability} to get
\[
  \E\left[\E_\xi[|U(\xi+r)|e^{m\beta\xi}]\right]^{4k}
  \le \E\left[\E_\xi[|U(\xi+r)|^{4k}e^{4km\beta\xi}]\right]
  \le K,
\]
and thus
\begin{equation}\label{eq:Aprime-over-B}
  \E\left|\frac{A'}{B}\right|^{4k}\le K.
\end{equation}
Lastly, we see
\[
  \left|\frac{B-B'}{B'}\right|
  \le \sum_{m=-S}^{S}\frac{e^{mh+Dm^2}|B(m)|}{B'}
  \le \sum_{m=-S}^{S}|B(m)|e^{-m\beta r}.
\]
For each term, we get
\[
  \E\big[|B(m)|e^{-m\beta r}\big]^{4k}
  \le \big(\E|B(m)|^{8k}\big)^{1/2}\big(\E e^{-8km\beta r}\big)^{1/2}
  \le \frac{K}{N^{2k}}.
\]
Summing over $m$ gives
\begin{equation}\label{eq:B-over-Bprime}
  \E\left|\frac{B-B'}{B'}\right|^{4k}\le \frac{K}{N^{2k}}.
\end{equation}
Using \eqref{eq:Aprime-over-B} and \eqref{eq:B-over-Bprime}, we get
\[
  \E\left[\left(\frac{A'}{B}\right)^{2k}\left(\frac{B'-B}{B'}\right)^{2k}\right]
  \le \left(\E\left|\frac{A'}{B}\right|^{4k}\right)^{1/2}
      \left(\E\left|\frac{B-B'}{B'}\right|^{4k}\right)^{1/2}
  \le \frac{K}{N^k}.
\]
Combining this with \eqref{eq:A-over-B-bound}, we obtain the proposition.
\end{proof}

\begin{lemma}\label{lem:onsager}
For every $k\ge1$,
\begin{equation}\label{eq:onsager-bound}
  \E|r_- -\gamma_N|^{2k}\le \frac{C}{N^k}.
\end{equation}
\end{lemma}

\begin{proof}
Let
\[
  F_-(x):=\frac{\Wtil_-'(\beta_-x+h)}{\Wtil_-(\beta_-x+h)}.
\]
Since the spins are bounded by $S$, we have $|F_-(x)|\le S$. The cavity decomposition gives
\[
  \langle\sigma_N\rangle
  =\frac{\langle W'(\beta l+h)\rangle_-}{\langle W(\beta l+h)\rangle_-}
  =\frac{\langle W'(\beta_- l_-+h)\rangle_-}{\langle W(\beta_- l_-+h)\rangle_-},
\]
and
\[
  \langle l_N\rangle
  =\sqrt{\frac{N-1}{N}}\,
  \frac{\langle l_-W(\beta_-l_-+h)\rangle_-}{\langle W(\beta_-l_-+h)\rangle_-}.
\]
Apply Proposition~\ref{prop:ratio} to the $(N-1)$-spin system, first with
\[
  U_1(x):=\frac{W'(\beta_-x+h)}{W(\beta_-x+h)},
\]
for which $U_1(x)W(\beta_-x+h)=W'(\beta_-x+h)$ and
\[
  \frac{\E_{\xi_-}W'(\beta_-(r_-+\xi_-)+h)}{\Wtil_-(\beta_-r_-+h)}
  =F_-(r_-),
\]
and then with $U(x)=x$. Using Theorem~\ref{thm:cavity} inside the ratio estimate, there exist error variables $\Delta_1$ and $\Delta_2$ such that
\begin{equation}\label{eq:delta1}
  \langle\sigma_N\rangle=F_-(r_-)+\Delta_1,
  \qquad
  \E|\Delta_1|^{2k}\le \frac{C}{N^k},
\end{equation}
and
\begin{equation}\label{eq:delta2}
  \frac{\langle l_-W(\beta_-l_-+h)\rangle_-}{\langle W(\beta_-l_-+h)\rangle_-}
  =r_-+\beta_-v_-F_-(r_-)+\Delta_2,
  \qquad
  \E|\Delta_2|^{2k}\le \frac{C}{N^k}.
\end{equation}
The second identity uses the Gaussian integration-by-parts formula
\[
  \E_{\xi_-}[(r_-+\xi_-)W(\beta_-(r_-+\xi_-)+h)]
  =r_-\Wtil_-(\beta_-r_-+h)+\beta_-v_-\Wtil_-'(\beta_-r_-+h).
\]
Consequently, writing $c_N=\sqrt{(N-1)/N}$, equations \eqref{eq:delta1} and \eqref{eq:delta2} give
\begin{equation}\label{eq:gamma-r-explicit}
  \gamma_N-r_-=(c_N-1)r_-+(c_N\beta_-v_- -\beta v)F_-(r_-)+c_N\Delta_2-\beta v\Delta_1.
\end{equation}
By Lemma~\ref{lem:param-comparison},
\[
  |c_N-1|+|c_N\beta_-v_- -\beta v|\le \frac{C}{N}.
\]
Moreover $r_-$ has Gaussian moments of all orders uniformly in $N$, since it is a Gaussian linear combination of bounded magnetizations. Since $|F_-(r_-)|\le S$, \eqref{eq:gamma-r-explicit}, the elementary inequality $|x_1+\cdots+x_4|^{2k}\le C_k\sum_i|x_i|^{2k}$, and \eqref{eq:delta1}-\eqref{eq:delta2} imply
\[
  \E|\gamma_N-r_-|^{2k}
  \le C\left(\frac{\E|r_-|^{2k}}{N^{2k}}+\frac1{N^{2k}}+\E|\Delta_1|^{2k}+\E|\Delta_2|^{2k}\right)
  \le \frac{C}{N^k}.
\]
This proves the lemma.
\end{proof}

\begin{lemma}\label{lem:center-integrability}
There exists $a_0>0$, depending only on $S,D,h$ and the high-temperature bound, such that
\begin{equation}\label{eq:center-integrability}
  \sup_{N\ge2}\E\exp(a_0(r_-^2+\gamma_N^2))<\infty.
\end{equation}
\end{lemma}

\begin{proof}
Conditional on the $(N-1)$-spin disorder, $r_-$ is a centered Gaussian variable with variance
\[
  \frac1{N-1}\sum_{j<N}\langle\sigma_j\rangle_-^2\le S^2,
\]
so $r_-$ has a uniform Gaussian-square exponential moment for sufficiently small $a_0$.

For $\gamma_N$, let $l=l_N$ and let $\mu_-$ be the tilted probability measure on the first $N-1$ spins proportional to $W(\beta l+h)dG_-$. The cavity identity gives $\langle l\rangle=\mu_-(l)$. By Jensen's inequality and the fact that the $m=0$ term of $W$ is one,
\[
  e^{2a_0\langle l\rangle^2}
  \le \frac{\langle e^{2a_0l^2}W(\beta l+h)\rangle_-}
  {\langle W(\beta l+h)\rangle_-}
  \le \langle e^{2a_0l^2}W(\beta l+h)\rangle_-.
\]
For a fixed spin configuration, $l$ is centered Gaussian with variance
\[
  s_\sigma^2=\frac1N\sum_{j<N}\sigma_j^2\le S^2.
\]
If $L\sim N(0,s_\sigma^2)$, then for $4a_0S^2<1$,
\[
  \E e^{2a_0L^2+m\beta L}
  =\frac1{\sqrt{1-4a_0s_\sigma^2}}
  \exp\left\{\frac{m^2\beta^2s_\sigma^2}{2(1-4a_0s_\sigma^2)}\right\},
\]
which is uniformly bounded for $|m|\le S$ and $\beta\le\beta_0$. Expanding $W$ and averaging over the first $N-1$ spins therefore gives
\[
  \sup_N\E e^{2a_0\langle l\rangle^2}<\infty.
\]
Finally, $0\le v=p-q\le S^2$ and $|\langle\sigma_N\rangle|\le S$, so
\[
  |\gamma_N-\langle l\rangle|\le \beta_0S^3.
\]
Reducing $a_0$ if necessary and combining the preceding bounds proves \eqref{eq:center-integrability}.
\end{proof}

\begin{lemma}\label{lem:lminus-l}
Under the same conditions as in Proposition~\ref{prop:ratio},
\begin{equation}\label{eq:lminus-l}
  \E\left[\frac{\langle U(l_-)W(\beta_-l_-+h)\rangle_-}{\langle W(\beta_-l_-+h)\rangle_-}-\langle U(l)\rangle\right]^{2k}\le \frac{K}{N^k}.
\end{equation}
\end{lemma}

\begin{proof}
The cavity decomposition gives the exact identity
\[
  \langle U(l)\rangle=\frac{\langle U(l)W(\beta l+h)\rangle_-}{\langle W(\beta l+h)\rangle_-}.
\]
Since $\beta_-l_-=\beta l$,
\[
  \frac{\langle U(l_-)W(\beta_-l_-+h)\rangle_-}{\langle W(\beta_-l_-+h)\rangle_-}-\langle U(l)\rangle
  =\frac{\langle (U(l_-)-U(l))W(\beta l+h)\rangle_-}{\langle W(\beta l+h)\rangle_-}.
\]
Let $d\mu_-$ be the tilted probability measure proportional to $W(\beta l+h)d G_-$. Jensen's inequality under this probability measure gives
\[
  \left|\frac{\langle (U(l_-)-U(l))W(\beta l+h)\rangle_-}{\langle W(\beta l+h)\rangle_-}\right|^{2k}
  \le
  \frac{\langle |U(l_-)-U(l)|^{2k}W(\beta l+h)\rangle_-}{\langle W(\beta l+h)\rangle_-}.
\]
Since $l=\sqrt{(N-1)/N}\,l_-$,
\[
  |l_- -l|\le \frac{C}{N}|l_-|.
\]
The mean-value theorem gives
\[
  |U(l_-)-U(l)|^{2k}\le \frac{C}{N^{2k}}|l_-|^{2k}\sup_{x\in[l,l_-]}|U'(x)|^{2k}.
\]
Because $U'$ has moderate growth and $l=\sqrt{(N-1)/N}\,l_-$, the function
\[
  M_N(x):=|x|^{2k}\sup_{y\text{ between }\sqrt{(N-1)/N}\,x\text{ and }x}|U'(y)|^{2k}
\]
has moderate growth uniformly in $N$. Since the denominator of the tilted ratio is at least one, it is enough to bound $\E\langle M_N(l_-)W(\beta_-l_-+h)\rangle_-$. Conditional on a spin configuration, $l_-$ is Gaussian with variance at most $S^2$, expanding $W$ and using the moderate growth of $M_N$ gives a uniform bound. Therefore the $2k$-th moment of the last display is bounded by $C/N^{2k}$, which is stronger than the claimed $C/N^k$ bound.
\end{proof}

\begin{lemma}\label{lem:replace-r-gamma-weight}
Under the same conditions as in Proposition~\ref{prop:ratio}, we have
\begin{equation}\label{eq:replace-r-gamma-weight}
  \E\left[\E_\xi\left[U(\xi+r_-)
  \left(\frac{W(\beta_-(\xi+r_-)+h)}{\Wtil_-(\beta_-r_-+h)}
  -\frac{W(\beta_-(\xi+\gamma)+h)}{\Wtil_-(\beta_-\gamma+h)}\right)\right]\right]^{2k}\le \frac{K}{N^k}.
\end{equation}
\end{lemma}

\begin{proof}
Let us define the function
\[
  f(x)=\frac{W(\beta_-(\xi+x)+h)}{\Wtil_-(\beta_-x+h)}.
\]
Let $y_1=\beta_-(\xi+x)+h$ and $y_2=\beta_-x+h$. Applying the quotient rule, we have
\[
  f'(x)=\beta_-\frac{W'(y_1)}{\Wtil_-(y_2)}-
  \beta_-\frac{W(y_1)\Wtil_-'(y_2)}{\Wtil_-(y_2)^2}.
\]
We have $|W'(y)|\le S W(y)$ and $|\Wtil_-'(y)|\le S\Wtil_-(y)$. Applying the triangle inequality yields
\begin{equation}\label{eq:fprime-bound-1}
  |f'(x)|\le 2S\beta_-\frac{W(y_1)}{\Wtil_-(y_2)}.
\end{equation}
Notice that $y_1=y_2+\beta_-\xi$. Using $\cosh(A+B)\le \cosh(A)e^{|B|}$, we have
\[
  W(y_1)\le e^{S\beta_-|\xi|}W(y_2).
\]
Since $p_-\ge q_-$ by Jensen's inequality, $\Dtil_-=D+\beta_-^2(p_- -q_-)/2\ge D$, so $\Wtil_-(x)\ge W(x)$ for all $x$. Combining these facts,
\[
  \frac{W(y_1)}{\Wtil_-(y_2)}\le \frac{W(y_1)}{W(y_2)}\le e^{S\beta_-|\xi|}.
\]
We now get
\[
  |f'(x)|\le 2S\beta_- e^{S\beta_-|\xi|}.
\]
By the mean value theorem,
\[
  |f(r_-)-f(\gamma)|\le 2S\beta_-|r_- -\gamma|e^{S\beta_-|\xi|}.
\]
Thus
\begin{align*}
  \E\left[\E_\xi[|U(\xi+r_-)| |f(r_-)-f(\gamma)|]\right]^{2k}
  &\le (2S\beta_-)^{2k}\E\left[|r_- -\gamma|^{2k}
      \left(\E_\xi[|U(\xi+r_-)|e^{S\beta_-|\xi|}]\right)^{2k}\right]\\
  &\le (2S\beta_-)^{2k}\left(\E|r_- -\gamma|^{4k}\right)^{1/2}\\
  &\qquad\times
      \left(\E\left[\left(\E_\xi[|U(\xi+r_-)|e^{S\beta_-|\xi|}]\right)^{4k}\right]\right)^{1/2}\\
  &\le \frac{K}{N^k},
\end{align*}
where the last inequality follows from Lemma~\ref{lem:onsager} and the Gaussian integrability bound in Lemma~\ref{lem:chen-integrability}.
\end{proof}

\begin{lemma}\label{lem:replace-U-center}
Under the same conditions as in Proposition~\ref{prop:ratio}, we have
\begin{equation}\label{eq:replace-U-center}
  \E\left[\E_\xi\left[(U(\xi+r_-)-U(\xi+\gamma))
  \frac{W(\beta_-(\xi+\gamma)+h)}{\Wtil_-(\beta_-\gamma+h)}\right]\right]^{2k}\le \frac{K}{N^k}.
\end{equation}
\end{lemma}

\begin{proof}
Let $f(x):=U(\xi+x)$. Then $f'(x)=U'(\xi+x)$. By the mean value theorem,
\[
  |U(\xi+r_-)-U(\xi+\gamma)|\le |r_- -\gamma|\sup_{x\in[r_-,\gamma]}|U'(\xi+x)|.
\]
Therefore,
\begin{align*}
&\left|\E_\xi\left[(U(\xi+r_-)-U(\xi+\gamma))
\frac{W(\beta_-(\xi+\gamma)+h)}{\Wtil_-(\beta_-\gamma+h)}\right]\right|\\
&\qquad\le |r_- -\gamma|\E_\xi\left[\sup_{x\in[r_-,\gamma]}|U'(\xi+x)|
\frac{W(\beta_-(\xi+\gamma)+h)}{\Wtil_-(\beta_-\gamma+h)}\right].
\end{align*}
Since $\Wtil_-(\beta_-\gamma+h)\ge W(\beta_-\gamma+h)$ and $\cosh(a+b)\le \cosh(a)e^{|b|}$,
\[
  \frac{W(\beta_-(\xi+\gamma)+h)}{\Wtil_-(\beta_-\gamma+h)}\le e^{S\beta_-|\xi|}.
\]
Taking the $2k$-th moment and applying Holder's inequality,
\begin{align*}
&\E\left[\E_\xi\left[(U(\xi+r_-)-U(\xi+\gamma))\frac{W(\beta_-(\xi+\gamma)+h)}{\Wtil_-(\beta_-\gamma+h)}\right]\right]^{2k}\\
&\qquad\le \left(\E|r_- -\gamma|^{4k}\right)^{1/2}
\left(\E\left[\left(\E_\xi\left[\sup_{x\in[r_-,\gamma]}|U'(\xi+x)|e^{S\beta_-|\xi|}\right]\right)^{4k}\right]\right)^{1/2}.
\end{align*}
The first factor is bounded by $K/N^k$ by Lemma~\ref{lem:onsager} applied with moment $4k$. For the second factor, moderate growth implies that, for every sufficiently small $a>0$, the integrand is bounded by a constant times $\exp\{a(\xi^2+r_-^2+\gamma^2)\}$. Lemma~\ref{lem:center-integrability} and the Gaussian exponential moments of $\xi$ therefore give a uniform bound. Hence \eqref{eq:replace-U-center} follows.
\end{proof}

\begin{lemma}\label{lem:denom-compare}
Under the same conditions as in Proposition~\ref{prop:ratio}, we have
\begin{equation}\label{eq:denom-compare}
  \E\left[\left(\frac1{\Wtil_-(\beta_-\gamma+h)}-\frac1{\Wtil(\beta\gamma+h)}\right)
  \E_\xi[U(\xi+\gamma)W(\beta_-(\xi+\gamma)+h)]\right]^{2k}\le \frac{K}{N^k}.
\end{equation}
\end{lemma}

\begin{proof}
We can write the expression inside the expectation as a product
\[
  \frac{\Wtil(\beta\gamma+h)-\Wtil_-(\beta_-\gamma+h)}{\Wtil(\beta\gamma+h)}\cdot
  \frac{\E_\xi[U(\xi+\gamma)W(\beta_-(\xi+\gamma)+h)]}{\Wtil_-(\beta_-\gamma+h)},
\]
and bound each factor separately. Using $\Wtil_-(\beta_-\gamma+h)\ge W(\beta_-\gamma+h)$ and $\cosh(a+b)\le \cosh(a)e^{|b|}$, we get
\[
  \frac{W(\beta_-(\xi+\gamma)+h)}{\Wtil_-(\beta_-\gamma+h)}
  \le \frac{W(\beta_-(\xi+\gamma)+h)}{W(\beta_-\gamma+h)}
  \le e^{S\beta_0|\xi|}.
\]
Therefore
\[
  \frac{|\E_\xi[U(\xi+\gamma)W(\beta_-(\xi+\gamma)+h)]|}{\Wtil_-(\beta_-\gamma+h)}
  \le \E_\xi[|U(\xi+\gamma)|e^{S\beta_0|\xi|}].
\]
By Jensen's inequality, the moderate growth of $U$, Lemma~\ref{lem:center-integrability}, and the Gaussian exponential moments of $\xi$, we get
\begin{equation}\label{eq:UW-moment}
  \E\left[\left(\frac{|\E_\xi[U(\xi+\gamma)W(\beta_-(\xi+\gamma)+h)]|}{\Wtil_-(\beta_-\gamma+h)}\right)^{4k}\right]\le K.
\end{equation}
Define
\[
  p_m(\gamma,h):=\frac{e^{\Dtil m^2+m(\beta\gamma+h)}}{\Wtil(\beta\gamma+h)},\qquad m\in\{-S,\ldots,S\}.
\]
Since the factors $e^{mh}$ cancel between numerator and denominator,
\[
  \frac{\Wtil_-(\beta_-\gamma+h)}{\Wtil(\beta\gamma+h)}
  =\sum_{m=-S}^{S}p_m(\gamma,h)e^{(\Dtil_- -\Dtil)m^2+m(\beta_- -\beta)\gamma}.
\]
Therefore
\[
  \frac{\Wtil(\beta\gamma+h)-\Wtil_-(\beta_-\gamma+h)}{\Wtil(\beta\gamma+h)}
  =\sum_{m=-S}^{S}p_m(\gamma,h)\left(1-e^{(\Dtil_- -\Dtil)m^2+m(\beta_- -\beta)\gamma}\right).
\]
By Lemma~\ref{lem:param-comparison}, $|\Dtil-\Dtil_-|\le C/N$ and $|\beta-\beta_-|\le C/N$, so for each $m$,
\[
  |(\Dtil_- -\Dtil)m^2+m(\beta_- -\beta)\gamma|\le \frac{C(1+|\gamma|)}{N}.
\]
Using $|1-e^x|\le |x|e^{|x|}$ and $\sum_m p_m=1$,
\begin{equation}\label{eq:denom-factor}
  \left|\frac{\Wtil(\beta\gamma+h)-\Wtil_-(\beta_-\gamma+h)}{\Wtil(\beta\gamma+h)}\right|
  \le \frac{C(1+|\gamma|)}{N}e^{C(1+|\gamma|)/N}.
\end{equation}
Since $e^{C(1+|\gamma|)/N}\le e^{C(1+|\gamma|)}$, Lemma~\ref{lem:center-integrability} gives
\begin{equation}\label{eq:denom-factor-moment}
  \E\left[\left(\frac{C(1+|\gamma|)}{N}e^{C(1+|\gamma|)/N}\right)^{4k}\right]\le \frac{K}{N^{4k}}.
\end{equation}
By Holder's inequality, \eqref{eq:UW-moment}, and \eqref{eq:denom-factor-moment}, we obtain \eqref{eq:denom-compare}.
\end{proof}

\begin{lemma}\label{lem:beta-compare}
Under the same conditions as in Proposition~\ref{prop:ratio}, we have
\begin{equation}\label{eq:beta-compare}
  \E\left[\E_\xi\left[U(\xi+\gamma)
  \left(\frac{W(\beta_-(\xi+\gamma)+h)}{\Wtil(\beta\gamma+h)}
  -\frac{W(\beta(\xi+\gamma)+h)}{\Wtil(\beta\gamma+h)}\right)\right]\right]^{2k}\le \frac{K}{N^k}.
\end{equation}
\end{lemma}

\begin{proof}
For $x\in[\beta_-,\beta]$, define $f(x):=W(x(\xi+\gamma)+h)$. Then
\[
  |f'(x)|\le S|\xi+\gamma|W(x(\xi+\gamma)+h).
\]
Since $x\in[\beta_-,\beta]$ and $\beta\le\beta_0$, using $\cosh(a+b)\le\cosh(a)e^{|b|}$, we get
\[
  W(x(\xi+\gamma)+h)\le e^{S\beta_0|\xi|}W(x\gamma+h).
\]
We do not use monotonicity of $W$ in its argument, since that monotonicity is false in general when $h$ and $\gamma$ have opposite signs. Instead, because $|x-\beta|\le C/N$,
\[
  W(x\gamma+h)\le e^{S|x-\beta||\gamma|}W(\beta\gamma+h)\le e^{CS|\gamma|/N}W(\beta\gamma+h).
\]
As $\Wtil(\beta\gamma+h)\ge W(\beta\gamma+h)$, it follows that
\[
  \frac{|f'(x)|}{\Wtil(\beta\gamma+h)}\le S|\xi+\gamma|e^{S\beta_0|\xi|}e^{CS|\gamma|/N}.
\]
By the mean value theorem,
\[
  \left|\frac{W(\beta_-(\xi+\gamma)+h)}{\Wtil(\beta\gamma+h)}-
  \frac{W(\beta(\xi+\gamma)+h)}{\Wtil(\beta\gamma+h)}\right|
  \le |\beta-\beta_-|S|\xi+\gamma|e^{S\beta_0|\xi|}e^{CS|\gamma|/N}.
\]
Therefore,
\begin{align*}
&\left|\E_\xi\left[U(\xi+\gamma)
  \left(\frac{W(\beta_-(\xi+\gamma)+h)}{\Wtil(\beta\gamma+h)}
  -\frac{W(\beta(\xi+\gamma)+h)}{\Wtil(\beta\gamma+h)}\right)\right]\right|\\
&\qquad\le |\beta-\beta_-|\E_\xi[|U(\xi+\gamma)|S|\xi+\gamma|e^{S\beta_0|\xi|}e^{CS|\gamma|/N}].
\end{align*}
Taking the $2k$-th moment and applying Holder's inequality, together with $|\beta-\beta_-|\le K/N$, Lemma~\ref{lem:center-integrability}, and the Gaussian exponential moments of $\xi$, gives the stronger estimate $K/N^{2k}$, hence \eqref{eq:beta-compare}.
\end{proof}

\begin{lemma}\label{lem:variance-comparison}
Let $\xi_-\sim N(0,v_-)$ and $\xi\sim N(0,v)$ be independent of the disorder. Then
\begin{equation}\label{eq:variance-comparison}
  \E\left|\frac{\E_{\xi_-}[U(\xi_-+\gamma)W(\beta(\xi_-+\gamma)+h)]}{\Wtil(\beta\gamma+h)}
  -\frac{\E_{\xi}[U(\xi+\gamma)W(\beta(\xi+\gamma)+h)]}{\Wtil(\beta\gamma+h)}\right|^{2k}\le \frac{K}{N^k}.
\end{equation}
\end{lemma}

\begin{proof}
By Lemma~\ref{lem:param-comparison}, $|v-v_-|\le C/N$. Let $Z\sim N(0,1)$ be independent of the disorder and use the common-normal coupling
\[
  \xi_-=\sqrt{v_-}\,Z,\qquad \xi=\sqrt v\,Z.
\]
This coupling also covers the case in which one of the variances is zero, and
\begin{equation}\label{eq:sqrt-variance-coupling}
  |\xi_- -\xi|=|\sqrt{v_-}-\sqrt v|\,|Z|
  \le \sqrt{|v_--v|}\,|Z|\le \frac{C}{\sqrt N}|Z|.
\end{equation}
Apply the mean-value theorem to
\[
  x\longmapsto U(x+\gamma)W(\beta(x+\gamma)+h).
\]
Since $|W'(y)|\le S W(y)$ and $\Wtil(\beta\gamma+h)\ge W(\beta\gamma+h)$, for every $x$ on the segment joining $\xi_-$ and $\xi$,
\[
  \frac{\left|\frac{d}{dx}\left[U(x+\gamma)W(\beta(x+\gamma)+h)\right]\right|}
  {\Wtil(\beta\gamma+h)}
  \le \bigl(|U'(x+\gamma)|+\beta S|U(x+\gamma)|\bigr)e^{S\beta|x|}.
\]
Under the common-normal coupling, $|x|\le S|Z|$. Thus the right-hand side is bounded by a moderate-growth function of $Z$ and $\gamma$. Lemma~\ref{lem:center-integrability} and the Gaussian exponential moments of $Z$ give uniform moments of this majorant. Taking the $2k$-th moment in \eqref{eq:sqrt-variance-coupling} therefore yields the required bound $K/N^k$.
\end{proof}

\begin{lemma}\label{lem:final-local}
Under the same conditions as in Proposition~\ref{prop:ratio}, we have
\begin{equation}\label{eq:final-local}
  \E\left[\langle U(l)\rangle-\frac{\E_\xi[U(\xi+\gamma)W(\beta(\xi+\gamma)+h)]}{\Wtil(\beta\gamma+h)}\right]^{2k}\le \frac{K}{N^k},
\end{equation}
where now $\xi\sim N(0,v)$ and $v=p-q$. Moreover,
\begin{equation}\label{eq:mixture-identity}
  \frac{\E_\xi[U(\xi+\gamma)W(\beta(\xi+\gamma)+h)]}{\Wtil(\beta\gamma+h)}
  =\int U(x)d\nu_N(x),
\end{equation}
where $d\nu_N$ is the mixture of $2S+1$ Gaussian measures defined in \eqref{eq:nu-def}.
\end{lemma}

\begin{proof}
Combining Proposition~\ref{prop:ratio}, Lemmas~\ref{lem:lminus-l}-\ref{lem:beta-compare}, and Lemma~\ref{lem:variance-comparison} gives the following step by step approximations,
\begin{align*}
A_0&:=\langle U(l)\rangle,
\\
A_1&:=\frac{\langle U(l_-)W(\beta_-l_-+h)\rangle_-}{\langle W(\beta_-l_-+h)\rangle_-}, \\ 
A_2&:=\frac{\E_{\xi_-}[U(r_-+\xi_-)W(\beta_-(r_-+\xi_-)+h)]}{\Wtil_-(\beta_-r_-+h)}, \\
A_3&:=\frac{\E_{\xi_-}[U(\gamma+\xi_-)W(\beta_-(\gamma+\xi_-)+h)]}{\Wtil_-(\beta_-\gamma+h)},\\
A_4&:=\frac{\E_{\xi_-}[U(\gamma+\xi_-)W(\beta(\gamma+\xi_-)+h)]}{\Wtil(\beta\gamma+h)},\\
A_5&:=\frac{\E_{\xi}[U(\gamma+\xi)W(\beta(\gamma+\xi)+h)]}{\Wtil(\beta\gamma+h)}.
\end{align*}
More explicitly, each adjacent difference satisfies
\begin{equation}\label{eq:adjacent-errors}
  \E|A_i-A_{i+1}|^{2k}\le \frac{C}{N^k},\qquad 0\le i\le 4.
\end{equation}
For $A_1-A_2$, this is Proposition~\ref{prop:ratio} applied to the $(N-1)$-spin system. The differences $A_2-A_3$, $A_3-A_4$, and $A_4-A_5$ are exactly the replacements controlled in Lemmas~\ref{lem:replace-r-gamma-weight}-\ref{lem:variance-comparison}. Therefore, by the triangle inequality and \eqref{eq:adjacent-errors},
\[
  \E|A_0-A_5|^{2k}\le \frac{C}{N^k},
\]
which proves \eqref{eq:final-local}.

It remains to identify the last expression. Recall $v=p-q$ and $\xi\sim N(0,v)$. Expanding $W$,
\[
  \E_\xi[U(\xi+\gamma)W(\beta(\xi+\gamma)+h)]
  =\sum_{m=-S}^{S}e^{Dm^2+m(\beta\gamma+h)}\E_\xi[U(\xi+\gamma)e^{m\beta\xi}].
\]
Rewriting the expectation on the r.h.s, we have
\[
  \E_\xi[U(\xi+\gamma)e^{m\beta\xi}]
  =e^{m^2\beta^2v/2}\int U(x)\varphi_{\gamma+m\beta v,v}(x)d x.
\]
Therefore
\[
  \frac{\E_\xi[U(\xi+\gamma)W(\beta(\xi+\gamma)+h)]}{\Wtil(\beta\gamma+h)}
  =\sum_{m=-S}^{S}\frac{e^{\Dtil m^2+m(\beta\gamma+h)}}{\Wtil(\beta\gamma+h)}
  \int U(x)\varphi_{\gamma+m\beta v,v}(x)d x,
\]
which is exactly $\int U(x)d\nu_N(x)$ by \eqref{eq:nu-def}.
\end{proof}

\begin{proof}[Proof of Theorem~\ref{thm:local}]
The theorem follows immediately from Lemma~\ref{lem:final-local}. The identification of the limiting measure is exactly \eqref{eq:mixture-identity}.
\end{proof}

\bibliographystyle{amsra}
\bibliography{GS.bib}

\end{document}